\documentclass{scrartcl}

\usepackage{amsmath}
 \numberwithin{equation}{section}
\usepackage{amssymb}
\usepackage{amsthm}
 \newtheorem{thm}{Theorem}[section]
 \newtheorem{cor}[thm]{Corollary}
 \newtheorem{lem}[thm]{Lemma}
\theoremstyle{definition}
 \newtheorem{defn}[thm]{Definition}
\theoremstyle{remark}
 \newtheorem{exam}[thm]{Example}
\usepackage[main=english]{babel}
\usepackage[noisbn,noissn,fixlanguage]{babelbib}
 \setbtxfallbacklanguage{english}
\usepackage{cite}
\usepackage[T1]{fontenc}
\usepackage{hyperref}
\usepackage{lmodern}
\usepackage{mathdots}
\usepackage{mathscinet}
\usepackage{mathtools}
 \mathtoolsset{mathic=true}
\usepackage{microtype}
\usepackage{url}

\usepackage{trigappr_arxiv}

\title{Duality results for a general trigonometric approximation problem}
\author{Lutz Klotz \and Conrad M\"adler}

\begin{document}
\maketitle

\begin{abstract}
 Let \(\alpha\in\iva\) and \(\mu\) be a regular finite Borel measure on a locally compact abelian group.
 The paper deals with a general trigonometric approximation problem in \(\Lam\), which arises in prediction theory of harmonizable symmetric \tas{} processes.
 To solve it, a duality method is applied, which is due to Nakazi and was generalized by Miamee and Pourahmadi and in the sequel successfully applied by several authors.
 The novelty of the present paper is that we do not make any additional assumption on \(\mu\).
 Moreover, for \(\alpha=2\), multivariate extensions are obtained.
\end{abstract}

\begin{description}
 \item[Keywords:] Regular Borel measure, space of \tai{} functions, trigonometric approximation, duality.
\end{description}

\section{Introduction}
 Let \(w\) be a weight function on \((-\pi,\pi]\), \(S\) a nonvoid subset of the set \(\Z\) of integers, \(s\in S\), and \(\TZS\) the linear space of trigonometric polynomials with frequencies from \(\ZS\).
 An important task in prediction theory of weakly stationary or, more generally, harmonizable symmetric \tas{} sequences is to compute the prediction error, \tie{}, the distance of the function \(\ec^{\iu s\cdot}\) to the set \(\TZS\) with respect to the metric of the Banach space \(\Law\), \(\alpha\in\iva\).
 In 1984 Nakazi~\zita{MR766702} introduced a new idea into the study of this problem.
 Among other things, his method opened a way for him to give an elegant proof of the celebrated Szeg\H{o} infimum formula.
 Miamee and Pourahmadi~\zita{MR949088} pointed out that the essence of Nakazi's technique is a certain duality between the spaces \(\Law\) and \(\Ldaw\), \(\da \defeq\alpha/(\alpha-1)\), presumed that \(w^\inv\) exists and is integrable.
 In the sequel this duality relation turned out to be rather fruitful.
 It was applied to a variety of sets \(S\), modified, and extended to more general prediction problems, \tcf{}\ the papers~\zitas{MR1443377,MR1230571,MR1289507,MR2262930}, as well as~\zita{MR1170451} for fields on \(\Z^2\),~\zita{MR772194} for processes on discrete abelian groups,~\zita{MR1928917} for multivariate sequences.
 Urbanik~\zita{MR1808671} defined a notion of a dual stationary sequence, \tcf{}~\zitaa{MR1849562}{\csec{8.5}} and~\zita{MR1899439}.
 Kasahara, Pourahmadi and Inoue~\zita{MR2547424} used a modified duality method to obtain series representations of the predictor.
 It should also be mentioned that \cthm{24} of~\zita{MR0009098} was, perhaps, the first published duality result of the type in question, \tcf{} its extension by Yaglom~\zitaa{MR0031214}{\cthm{2}}.

 Many of the preceding results were obtained under the rather strong additional condition that \(w^\inv\) exists and is integrable, although for some special sets authors succeeded in weakening this condition, \tcf{}~\zitas{MR1443377,MR772194,MR2262930}.
 In our paper we study the above described problem for spaces \(\Lam\), where the Borel measure \(\mu\) is not assumed to be absolutely continuous.
 Moreover, motivated by Weron's paper~\zita{MR772194} we shall be concerned with regular finite Borel measures on locally compact abelian groups.
 Since in the literature there exists differing definitions of a regular measure, \rsec{3} deals with the definition and some basic facts of regular measures.
 The results of \rsecss{2}{4} show that under a condition, which is satisfied by many sets \(S\) occurring in application, one can assume that the measure \(\mu\) is absolutely continuous.
 Establishing this result, we introduce a class of sets, which we call class of \tACs{s} and which, as it seems to us, deserves further investigation.
 \rsec{5} gives a solution to the problem if \(\mu\) is absolutely continuous.
 Unlike most of the authors above we do not make any condition on the corresponding Radon-Nikodym derivative.

 In~\zita{MR1462266} there were defined Banach spaces of matrix-valued functions \tai{} with respect to a positive semidefinite matrix-valued measure, see \rsec{2} for the definition and basic facts.
 Since part of our results can be easily generalized, we state and prove them in this more general framework.
 If \(S\) is a singleton, the corresponding results can be used to obtain minimality criteria for multivariate stationary sequences.
 We shall not go into detail but refer to the recent paper~\zita{KM15}, where various minimality notions were discussed.

 If \(X\) is a matrix, denote by \(X^\ad\), \(X^\mpi\), and \(\ran{X}\) its adjoint, Moore-Penrose inverse, and range, \tresp{} The symbol \(\Oe\) stands for the zero element in an arbitrary linear space.

\section{The space $\LaM$}\label{2}
 Let \(q\in\N\), \(\Mggq\) be the cone of all \tpsd{} (hence, \tH{}) \tqqa{matrices} with complex entries.
 Let \(\OA\) be a measurable space and \(M\) an \taval{\Mggq} measure on \(\A\).
 If \(\tau\) is a \tsf{} measure on \(\A\) such that \(M\) is absolutely continuous with respect to \(\tau\), denote by \(\dif M/\dif\tau\) the corresponding Radon-Nikodym derivative and by \(P(\omega)\) the orthoprojection in \(\Cq\) onto \(\ran{\rk{\dif M/\dif\tau}(\omega)}\), \(\omega\in\Omega\).
 Let \(\no{\cdot}\) be the euclidean norm on \(\Cq\) and write the vectors of \(\Cq\) as column vectors.
 Two \tAm{} \taval{\Cq} functions \(f\) and \(g\) are called \tMe{} if \(Pf=Pg\) \tae{\tau} For \(\alpha\in\iva\), denote by \(\LaM\) the space of all (\tMec{es} of) \tAm{} functions \(f\) such that \(\noa{f}\defeq\ek{\int_\Omega\no{\rk{\dif M/\dif\tau}^{1/\alpha}f}^\alpha\dif\tau}^{1/\alpha}<\infty\).
 Recall that the definition of \(\LaM\) does not depend on the choice of \(\tau\) and that \(\LaM\) is a Banach space with respect to the norm \(\noa{\cdot}\).
 Note that \(\LhM\) is a Hilbert space with inner product \(\int_\Omega g^\ad\rk{\dif M/\dif\tau}f\dif\tau\), \(f,g\in\LhM\), and that for \(q=1\), the space \(\LaM\) is the well known space of (equivalence classes of) \tAm{} \taval{\C} functions \tai{} with respect to \(M\).
 The space \(\LhM\) was introduced by I.~S.~Kats~\zita{MR0080280} and in a somewhat more general form by Rosenberg~\zita{MR0163346}.
 Both notions were applied in the theory of weakly stationary processes with the same success, \tcf{}~\zitas{MR0159363,MR0279952} for an application of Kats' and Rosenberg's definitions, \tresp{} An extension to \(\alpha\neq2\) was given by Duran and Lopez-Rodriguez~\zita{MR1462266}, \tcf{}~\zita{MR1160966} for a more general setting of operator-valued measures.
 To simplify the presentation slightly we shall be concerned with the space \(\LaM\) as defined above, which can be considered as a generalization of Kats' definition to the case \(\alpha\neq2\).

\begin{lem}[\tcf{}~\zitaa{MR1462266}{\cthm{2.5}},~\zitaa{MR1160966}{\cthm{9}}]\label{L2.1}
 Let \(\alpha\in\iva\) and \(\da \defeq\alpha/\rk{\alpha-1}\).
 If \(\ell\) is a bounded linear functional on \(\LaM\), then there exists \(g\in\LdaM\) such that \(\ell(f)=\int_\Omega g^\ad\rk{\dif M/\dif\tau}f\dif\tau\) for all \(f\in\LaM\).
 The correspondence \(\ell\mapsto g\) establishes an isometric isomorphism between the dual space of \(\LaM\) and the space \(\LdaM\).
\end{lem}

 Let \(M_1\) and \(M_2\) be \taval{\Mggq} measures on \(\A\) such that \(M=M_1+M_2\) and \(M_1(A)=M_2(\Omega\setminus A)=\Oe\) for some set \(A\in\A\).
 Let \(\ind{C}\) denote the indicator function of a set \(C\).
 Identifying \(\LaMa\) with the space \(\indOA \LaM=\setaa{\indOA f}{f\in\LaM}\) and \(\LaMb\) with \(\indA \LaM\), one obtains a direct sum decomposition \(\LaM=\LaMa\dotplus\LaMb\).
 For a linear subset \(\Lc\) of \(\LaM\), denote by \(\Lca\) its closure in \(\LaM\)

\begin{lem}\label{L2.2}
 If
 \begin{equation}\label{e2.1}
  \LaMb
  \subseteq\Lca,
 \end{equation}
 then
 \begin{equation}\label{e2.2}
  \cl{\rk{\indA \Lc}}
  =\LaMb
 \end{equation}
 and
 \begin{equation}\label{e2.3}
  \Lca
  =\cl{\rk{\indOA \Lc}}\dotplus\cl{\rk{\indA \Lc}}
  =\cl{\rk{\indOA \Lc}}\dotplus\LaMb.
 \end{equation}
\end{lem}
\begin{proof}
 The continuity of the map \(f\mapsto\indA f\), \(f\in\LaM\), and condition \eqref{e2.1} yield \(\LaMb=\indA\LaM\subseteq\indA\Lca\subseteq\cl{\rk{\indA\Lc}}\).
 Since the inclusion \(\cl{\rk{\indA\Lc}}\subseteq\LaMb\) is obvious, equality \eqref{e2.2} follows.
 For the proof of \eqref{e2.3}, note first that \(\Lca\subseteq\cl{\rk{\indOA\Lc}}\dotplus\cl{\rk{\indA\Lc}}\).
 To prove the opposite inclusion, let \(f\in\Lca\).
 By \eqref{e2.2} and \eqref{e2.1}, we have \(\indA f\in\indA\Lca\subseteq\cl{\rk{\indA\Lc}}=\LaMb\subseteq\Lca\), which gives \(\cl{\rk{\indA\Lc}}\subseteq\cl{\rk{\indA\Lca}}\subseteq\Lca\).
 The relation \(\indOA f=f-\indA f\in\Lca\) implies that \(\cl{\rk{\indOA\Lc}}\subseteq\Lca\), hence, \(\cl{\rk{\indOA \Lc}}\dotplus\cl{\rk{\indA \Lc}}\subseteq\Lca\).
\end{proof}

\begin{lem}\label{L2.3}
 Let \(f\in\LaM\) and \(\rho\defeq\inf\setaa{\noa{f-g}}{g\in\Lc}\), \(\rho_1\defeq\inf\setaa{\noa{\indOA(f-g)}}{g\in\Lc}\).
 If \eqref{e2.1} is satisfied, then \(\rho=\rho_1\).
\end{lem}
\begin{proof}
 For \(\epsilon>0\), there exist \(g_1,g_2\in\Lc\) such that \(\noa{\indOA(f-g_1)}<\rho_1+\epsilon\) and \(\noa{\indA(f-g_2)}<\epsilon\) according to \eqref{e2.2}.
 By \rlem{L2.2}, \(\indOA g_1+\indA g_2\in\Lca\).
 Therefore, \(\rho\leq\noa{f-\rk{\indOA g_1+\indA g_2}}\leq\noa{\indOA(f-g_1)}+\noa{\indA(f-g_2)}<\rho_1+2\epsilon\), hence, \(\rho\leq\rho_1\).
 Since the inequality \(\rho_1\leq\rho\) is trivial, the result follows.
\end{proof}

\section{Regular Borel measures on locally compact abelian groups}\label{3}
 Let \(\Gamma\) be a locally compact abelian group, \(\BsaG\) the \tsa{} of Borel sets of \(\Gamma\), and \(\lambda\) a Haar measure on \(\BsaG\).
 We recall the definition and some elementary properties of regular measures, however, we mention that we do not know an example of a non-regular finite Borel measure on a locally compact abelian group.
 Therefore, some of the following assertions might be redundant.

 A finite non-negative measure \(\mu\) on \(\BsaG\) is called
\begin{itemize}
 \item[a)] \emph{outer regular} if \(\mu(B)=\inf\setaa{\mu(U)}{U\text{ is open and }B\subseteq U\subseteq\Gamma}\) for all \(B\in\BsaG\),
 \item[b)] \emph{inner regular} if \(\mu(B)=\sup\setaa{\mu(K)}{K\text{ is compact and }K\subseteq B}\) for all \(B\in\BsaG\),
 \item[c)] \emph{regular} if it is outer regular and inner regular,
\end{itemize}
\tcf{}~\zitaa{MR578344}{\cpage{206}} and~\zitaa{MR0156915}{(11.34)}.

\begin{lem}\label{L3.1}
 \begin{aeqi}{0}
  \item\label{L3.1.i} Let \(\mu\) be a finite non-negative measure on \(\BsaG\).
 If it is inner regular, then it is outer regular.
  \item\label{L3.1.ii} Any positive linear combination of regular finite non-negative measures is regular.
  \item\label{L3.1.iii} Let \(\mu\) and \(\nu\) be finite non-negative measures on \(\BsaG\).
 If \(\mu\) is regular and \(\nu\) is absolutely continuous with respect to \(\mu\), then \(\nu\) is regular.
 \end{aeqi}
\end{lem}
\begin{proof}
 To prove~\ref{L3.1.i} let \(B\in\BsaG\) and \(\co{B}\defeq\Gamma\setminus B\).
 If \(\mu\) is inner regular, for \(\epsilon>0\) there exists a compact set \(K\subseteq\co{B}\) such that \(\mu(\co{B}\setminus K)<\epsilon\).
 Since \(\Gamma\) is a Hausdorff space, the compact set \(K\) is closed, hence, the set \(\co{K}\) is open, and it satisfies \(B\subseteq\co{K}\) and \(\mu(\co{K}\setminus B)=\mu(\co{B}\setminus K)<\epsilon\), which implies that \(\mu\) is outer regular.
 Assertion~\ref{L3.1.ii} is clear and assertion~\ref{L3.1.iii} is an immediate consequence of the following fact.
 If \(\nu\) is absolutely continuous with respect to \(\mu\), for \(\epsilon>0\) there exists \(\delta>0\) such that for all \(B\in\BsaG\), the inequality \(\mu(B)<\delta\) yields \(\nu(B)<\epsilon\), \tcf{}~\zitaa{MR578344}{\clem{4.2.1}}.
\end{proof}

 A \taval{\C} measure \(\mu\) on \(\BsaG\) is called \emph{regular} if its variation \(\var{\mu}\) is regular.
 It is called \emph{absolutely continuous} or \emph{singular}, \tresp{}, if its variation is absolutely  continuous or singular with respect to \(\lambda\).

\begin{lem}\label{L3.2}
 Let \(\mu\) be an absolutely continuous and regular \taval{\C} measure on \(\BsaG\).
 Then there exists a Radon-Nikodym derivative of \(\mu\) with respect to \(\lambda\).
\end{lem}
\begin{proof}
 Since \(\mu\) is assumed to be regular, there exists a sequence \(\gk{K_n}_{n\in\N}\) of compact sets such that \(\var{\mu}(\Gamma\setminus K_n)<1/n\), \(n\in\N\).
 If \(B\defeq\bigcup_{n=1}^\infty K_n\), then \(B\in\BsaG\), \(\var{\mu}(\Gamma\setminus B)=0\), and from \(\lambda(K_n)<\infty\), \(n\in\N\), we get that \(B\) is a set of \tsf{} \tam{\lambda}.
 Thus, the assertion follows from the Radon-Nikodym theorem, \tcf{}~\zitaa{MR578344}{\cthm{4.2.3}}.
\end{proof}

 By definition, an \taval{\Mggq} measure \(M\) on \(\BsaG\) is \emph{regular} if all its entries are regular.
 Let \(\mu_{jk}\) be the measure at place \((j,k)\).
 Since for all \(B\in\BsaG\), \(\abs{\mu_{jk}(B)}\leq\ek{\mu_{jj}(B)+\mu_{kk}(B)}/2\), hence \(\var{\mu_{jk}}(B)\leq\ek{\mu_{jj}(B)+\mu_{kk}(B)}/2\), from \rlem{L3.1} one can conclude that \(M\) is regular if and only if each measure on the principal diagonal is regular.
 Taking into account \rlemp{L3.1}{L3.1.iii} and \rlem{L3.2}, we arrive at the following result.

\begin{lem}\label{L3.3}
 If \(\ac{M}\) is the absolutely continuous part of a regular \taval{\Mggq} measure \(M\) on \(\BsaG\), then there exists the Radon-Nikodym derivative \(\dif\ac{M}/\dif\lambda\eqdef W\) of \(\ac{M}\) with respect to \(\lambda\).
\end{lem}

\section{A general trigonometric approximation problem in $\LaM$}\label{4}
 For \(k\in\mn{1}{q}\) denote by \(\eu{k}\) the \tth{k} vector of the standard orthonormal basis of \(\Cq\).
 Let \(G\) be a locally compact abelian group and \(\Gamma\) its dual.
 The value of \(\gamma\in\Gamma\) at \(x\in G\) is denoted by \(\inner{\gamma}{x}\).
 For \(x\in G\), define a function \(\chu{x}\) by \(\chua{x}{\gamma}\defeq\inner{\gamma}{x}\), \(\gamma\in\Gamma\), and set \(\chuu{x}{k}\defeq\chu{x}\eu{k}\).

 Let \(S\) be a nonvoid subset of \(G\).
 If \(\GsS\) is not empty, denote by \(\T{\GsS}\) the linear space of all \taval{\Cq} trigonometric polynomials on \(\Gamma\) with frequencies from \(\GsS\), \tie{}, \(t\in\T{\GsS}\) if and only if it has the form \(t=\sum_{j=1}^n\chu{x_j}u_j\), \(x_j\in\GsS\), \(u_j\in\Cq\), \(j\in\mn{1}{n}\), \(n\in\N\).
 If \(\GsS\) is empty, let \(\T{\GsS}\) be the space consisting of the zero function on \(\Gamma\).

 Let \(M\) be a regular \taval{\Mggq} measure on \(\BsaG\).
 Motivated by problems of prediction theory we are interested in computing the distance
\begin{align*}
 d&\defeq\inf\setaa*{\noa*{\chuu{s}{k}-t}}{t\in\T{\GsS}},&s&\in S,&k&\in\mn{1}{q}.
\end{align*}

\begin{thm}\label{T4.1}
 Let \(M=\ac{M}+\si{M}\) be the Lebesgue decomposition of \(M\) into its absolutely continuous part \(\ac{M}\) and singular part \(\si{M}\).
 If
 \begin{equation}\label{T4.1.1}
  \Loa{\alpha}{\si{M}}
  \subseteq\cl{\ek*{\T{\GsS}}},
 \end{equation}
 then for \(s\in S\), \(k\in\mn{1}{q}\),
 \[
  d
  =\inf\setaa*{\ek*{\int\no*{\rk*{\frac{\dif\ac{M}}{\dif\lambda}}^{1/\alpha}\rk{\chuu{s}{k}-t}}^\alpha\dif\lambda}^{1/\alpha}}{t\in\T{\GsS}},
 \]
 where here and in what follows the domain of integration is \(\Gamma\) if it is not indicated explicitly.
\end{thm}

 The preceding theorem immediately follows from \rlemss{L2.3}{L3.3} and it shows that under condition \eqref{T4.1.1} it is enough to solve the approximation problem for the absolutely continuous part of \(M\).
 The next section deals with its partial solution and the rest of the present section is devoted to subsets of \(G\) with the property that \eqref{T4.1.1} is satisfied for each regular \taval{\Mggq} measure.

 For a regular \taval{\Cq} measure \(\mu\) on \(\BsaG\), denote by \(\fsinv{\mu}\) its inverse Fourier-Stieltjes transform, \tie{}\ \(\fsinv{\mu}(x)\defeq\int\inner{\gamma}{x}\mu(\dif\gamma)=\int\chu{x}\dif\mu\), \(x\in G\).
 Let \(\LlCq\) be the Banach space of \tam{\BsaG}able \taval{\Cq} functions integrable with respect to \(\lambda\).
 If \(f\in\LlCq\), the symbol \(\fsinv{f}\) stands for its inverse Fourier transform.

\begin{defn}\label{D4.2}
 A subset \(S\) of \(G\) is called an \emph{\tACs{}}, if for all regular \taval{\Cq} measures \(\mu\) on \(\BsaG\), the equality \(\fsinv{\mu}(x)=0\) for all \(x\in S\) implies that \(\mu\) is absolutely continuous.
\end{defn}

\begin{lem}\label{L4.3}
 If \(\GsS\) is an \tACs{}, then \eqref{T4.1.1} is satisfied.
\end{lem}
\begin{proof}
 Assume that there exists \(f\in\Loa{\alpha}{\si{M}}\), which does not belong to \(\cl{\ek{\T{\GsS}}}\).
 Then by \rlem{L2.1} there exists \(g\in\LdaM\) such that
 \begin{align}\label{L4.3.1}
  \int\chuu{x}{k}^\ad\frac{\dif M}{\dif\tau}g\dif\tau&=0&\text{for all }x&\in\GsS\text{, }k\in\mn{1}{q},
 \end{align}
 and
 \begin{equation}\label{L4.3.2}
  \int f^\ad\frac{\dif M}{\dif\tau}g\dif\tau
  \neq0.
 \end{equation}
 From \eqref{L4.3.1} we can derive that the \taval{\Cq} measure \((\dif M/\dif\tau)g\dif\tau\) is absolutely continuous if \(\GsS\) is an \tACs{}.
 Since \(f\in\Loa{\alpha}{\si{M}}\), it follows \(\int f^\ad(\dif M/\dif\tau)g\dif\tau=0\), a contradiction to \eqref{L4.3.2}.
\end{proof}

 We conclude the section with some examples of \tACs{s}.
 A deeper study of this class of subsets would be of interest.

\begin{exam}\label{E4.4}
 \begin{itemize}
  \item[a)] If \(S\) is a compact subset of \(G\), then \(\GsS\) is an \tACs{}.
 For if \(\mu\) is a regular \taval{\Cq} measure on \(\BsaG\) and \(\fsinv{\mu}(x)=0\), \(x\in\GsS\), then the continuous function \(\fsinv{\mu}\) has compact support and, hence, is integrable with respect to a Haar measure on \(G\).
 Applying the inversion theorem, \tcf{}~\zitaa{MR0152834}{\cthm{1.5.1}}, and a uniqueness property of the inverse Fourier-Stieltjes transform, \tcf{}~\zitaa{MR0152834}{\cthm{1.3.6}}, we obtain that \(\mu\) is absolutely continuous.
  \item[b)] Each subset of a compact abelian group is an \tACs{}.
  \item[c)] If \(S\) is an \tACs{}, then \(-S\) and \(x+S\), \(x\in G\), are \tACs{s}, and if \(S\subseteq S_1\subseteq G\), then \(S_1\) is an \tACs{}.
  \item[d)] If \(G=\Z\) and \(S=\N\) or \(G=\R\) and \(S=[0,\infty)\), \tresp{}, then the Theorem of F.~and M.~Riesz implies that \(S\) is an \tACs{}, \tcf{}~\zitaa{MR0152834}{\cthmss{8.2.1}{8.2.7}}.
  \item[e)] A theorem of Bochner claims that the set of points of \(\Z^2\) which belong to a closed sector of the plane whose opening is larger than \(\pi\), is an \tACs{}, \tcf{}~\zitaa{MR0152834}{\cthm{8.2.5}}.
 Note that for the measure \(\mu\defeq\lambda\otimes\delta_0\) on \(\Bsa{(-\pi,\pi]\times(-\pi,\pi]}\), where \(\delta_0\) denotes the Dirac measure at \(0\), one has \(\fsinv{\mu}((m,n))=0\) for all \(m\in\Z\setminus\set{0}\), \(n\in\Z\).
 This shows that the lattice points of a half-plane do not form an \tACs{}.
 \end{itemize}
\end{exam}

\section{The case of an absolutely continuous measure}\label{5}
 Let \(\dif M=W\dif\lambda\), where \(W\) is a \tamable{\BsaG} \taval{\Mggq} function on \(\Gamma\) integrable with respect to \(\lambda\).
 For brevity, set \(\LaM\eqdef\LaW\).
 It is clear that the value of \(d\) does not depend on how large the \tMec{es} are chosen.
 Recall that \(P(\gamma)\) is the orthoprojection in \(\Cq\) onto \(\ran{W(\gamma)}\), \(\gamma\in\Gamma\).
 If \(f\) is a \tam{\BsaG}able \taval{\Cq} function, then \(f\) and \(Pf\) are \tMe{}.
 Therefore, each \tMec{} contains a function \(h\) such that
\begin{align}\label{E5.1}
 h(\gamma)&\in\Ran{P(\gamma)}&\text{for all }\gamma&\in\Gamma,
\end{align}
 and shrinking the \tMec{} we can and shall assume that for all functions \(h\) of the equivalence classes of \(\LaW\) relation \eqref{E5.1} is satisfied.

 The goal of the present section is to derive expressions for \(d\) if \(q=1\) and \(\alpha\in\iva\) or if \(q\in\N\) and \(\alpha=2\).
 Let us assume first that \(q=1\) and let us denote the scalar-valued weight function \(W\) by \(w\).
 Setting \(B\defeq\setaa{\gamma\in\Gamma}{w(\gamma)\neq0}\), one has
\[
 w^\mpi(\gamma)
 =
 \begin{cases}
  1/w(\gamma),&\text{if }\gamma\in B\\
  0,&\text{if }\gamma\in\Gamma\setminus B
 \end{cases}
\]
 and condition \eqref{E5.1} implies that if \(f\in\Law\), then \(f=0\) on \(\Gamma\setminus B\).
 The distance \(d\) can be written as \[\begin{split}
 d
 &=\inf\setaa*{\noa{\chu{s}-t}}{t\in\T{\GsS}}\\
 &=\inf\setaa*{\rk*{\int\abs{\chu{s}-t}^\alpha w\dif\lambda}^{1/\alpha}}{t\in\T{\GsS}}\\
 &=\inf\setaa*{\rk*{\int\abs{\ind{B}\chu{s}-t}^\alpha w\dif\lambda}^{1/\alpha}}{t\in\ind{B}\T{\GsS}},
\end{split}\]
\(s\in S\).
 For \(\alpha\in\iva\), let \(\da \defeq\alpha/(\alpha-1)\) and \(\beta\defeq1/(\alpha-1)\).

\begin{lem}\label{L5.1}
 For any bounded linear functional \(\ell\) on \(\Law\), there exists \(h\in\Ldawpb \) such that
 \begin{align*}
  \ell(f)&=\int f h^\ad\dif\lambda,&f&\in\Law.
 \end{align*}
 The mapping \(\ell\mapsto h\) establishes an isometric isomorphism between the dual space of \(\Law\) and the space \(\Ldawpb \).
\end{lem}
\begin{proof}
 If \(g\in\Loa{\da}{w}\), then \(\int\abs{gw}^\da \rk{w^\mpi}^\beta\dif\lambda=\int\abs{g}^\da w\dif\lambda\), which shows that the correspondence \(g\mapsto gw\), \(g\in\Loa{\da}{w}\), is an isometry from \(\Loa{\da}{w}\) into \(\Ldawpb \).
 Moreover, if \(h\in\Loa{\da}{\rk{w^\mpi}^\beta}\), then \(hw^\mpi w=h\) by \eqref{E5.1} and \(\int\abs{hw^\mpi}^\da w\dif\lambda=\int\abs{h}^\da\rk{w^\mpi}^\beta\dif\lambda<\infty\).
 Therefore, the correspondence \(g\mapsto gw\) maps \(\Loa{\da}{w}\) onto \(\Loa{\da}{\rk{w^\mpi}^\beta}\), and the lemma follows from the well known description of the dual space of \(\Law\), \tcf{}~\rlem{L2.1}.
\end{proof}

 The preceding lemma shows that if \(f\in\Law\) and \(h\in\Ldawpb \) satisfies \eqref{E5.1}, then
\[
 \int\abs*{fh^\ad}\dif\lambda
 \leq\noa{f}\ek*{\int\abs{h}^\da\rk{w^\mpi}^\beta\dif\lambda}^{1/\da}.
\]
 It follows that under condition \eqref{E5.1} the set
\[
 \D
 \defeq\setaa*{h\in\LoA{\da}{\rk{w^\mpi}^\beta}}{\fsinv{\rk{h^\ad}}(x)=0\text{ for all }x\in\GsS}
\]
 is defined correctly and that \(\int fh^\ad\dif\lambda=0\) if \(f\in\cl{\T{\GsS}}\) and \(h\in\D\).
 Taking into account a general approximation result in Banach spaces, \tcf{}~\zitaa{MR0268655}{\cthm{7.2}}, we can conclude that the following assertion is true.

\begin{lem}\label{L5.2}
 For any \(s\in S\), the distance \(d\) is equal to \(\sup\setaa{\abs{\fsinv{\rk{h^\ad}}(s)}}{h\in\D\text{ and }  \int\abs{h}^\da\rk{w^\mpi}^\beta\dif\lambda\leq1}\).
\end{lem}

 Let \(\Ds\defeq\setaa{h\in\D}{\fsinv{\rk{h^\ad}}(s)=1}\), \(s\in S\).

\begin{thm}\label{T5.3}
 Let \(q=1\), \(\alpha\in\iva\), and \(s\in S\).
 Then
 \begin{equation}\label{E5.2}
  d
  =\inf\setaa*{\noa{\chu{s}-t}}{t\in\T{\GsS}}
  =\sup\setaa*{\ek*{\int\abs{h}^\da\rk{w^\mpi}^\beta\dif\lambda}^{-1/\da}}{h\in\Ds},
 \end{equation}
 with the convention, that the right-hand side of \eqref{E5.2} is assumed to be \(0\) if \(\Ds\) is empty.
\end{thm}
\begin{proof}
 Define on the linear space \(\D\) two positive homogeneous and non-negative functionals \(\F\) and \(\G\) by \(\F(h)\defeq\ek{\int\abs{h}^\da\rk{w^\mpi}^\beta\dif\lambda}^{1/\da}\) and \(\G(h)\defeq\abs{\fsinv{\rk{h^\ad}}(s)}\), \(h\in\D\).
 If \(\G(h)=0\) for all \(h\in\D\), then the left-hand side of \eqref{E5.2} is \(0\) by \rlem{L5.2} and the right-hand side equals \(0\) by the convention made.
 If \(\G(h)\neq0\) for some \(h\in\D\), then \rlem{L5.2} and a well known duality relation, \tcf{}~\zitaa{MR0224158}{\clem{7.1}}, imply that the distance $d$ is equal to \(d=\sup\setaa{\G(h)}{h\in\D\text{ and }\F(h)\leq1}=\ek{\inf\setaa{\F(h)}{h\in\D\text{ and }\G(h)\geq1}}^\inv\).
 Since \(\ek{\inf\setaa{\F(h)}{h\in\D\text{ and }\G(h)\geq1}}^\inv={\ek{\inf\setaa{\F(h)}{h\in\D\text{ and }\G(h)=1}}^\inv}=\ek{\inf\setaa{\F(h)}{h\in\D\text{ and }\fsinv{\rk{h^\ad}}(s)=1}}^\inv=\sup\setaa{\ek{\F(h)}^\inv}{h\in\Ds}\), the theorem is proved.
\end{proof}

\begin{cor}\label{C5.4}
 The set \(\T{\GsS}\) is dense in \(\Law\) if and only if there does not exist a function \(h\in\D\setminus\set{0}\) such that \(\int\abs{h}^\da\rk{w^\mpi}^\beta\dif\lambda<\infty\).
\end{cor}
\begin{proof}
 Note that \(h\in\D\setminus\set{0}\) if and only if \(h\in\D\) and there exists \(s\in S\) such that \(\fsinv{\rk{h^\ad}}(s)\neq0\), which gives \(\D\setminus\set{0}=\bigcup\setaa{a\Ds}{a\in\C\setminus\set{0}\text{, }s\in S}\).
 If \(\delta\) denotes the right-hand side of \eqref{E5.2}, we obtain the following chain of equivalences: \(\T{\GsS}\) is dense in \(\Law\)\tarlr{}\(d=0\) for all \(s\in S\)\tarlr{}\(\delta=0\) for all \(s\in S\)\tarlr{}\(\int\abs{h}^\da\rk{w^\mpi}^\beta\dif\lambda=\infty\) for all \(h\in\D\setminus\set{0}\).
\end{proof}

 If \(q\) is an arbitrary positive integer, an analogous result to that of \rlem{L5.1} is not true in general.
 However, for \(\alpha=2\), the above method can be extended.
 We briefly sketch the main steps.
\begin{description}
 \item[Step 1:] The correspondence \(\ell\mapsto h\), defined by \(\ell(f)=\int h^\ad f\dif\lambda\), \(f\in\Loa{2}{W}\), establishes an isometric isomorphism between the dual space of \(\Loa{2}{W}\) and the space \(\Loa{2}{W^\mpi}\).
 \item[Step 2:] \(\int\abs{h^\ad f}\dif\lambda\leq\noq{f}\ek{\int h^\ad W^\mpi h\dif\lambda}^{1/2}\) for all \(f\in\Loa{2}{W}\) and \(h\in\Loa{2}{W^\mpi}\) satisfying \eqref{E5.1}.
 \item[Step 3:] The set \(\Dt\defeq\setaa{h\in\Loa{2}{W^\mpi}}{\fsinv{\rk{h^\ad}}(x)=0\text{ for }x\in\GsS}\) is defined correctly and \(\int h^\ad f\dif\lambda=0\) for all \(f\in\Loa{2}{W}\) and \(h\in\Loa{2}{W^\mpi}\) satisfying \eqref{E5.1}. 
 \item[Step 4:] For all \(s\in S\) and \(k\in\mn{1}{q}\), the distance \(d\) is equal to \(\sup\setaa{\abs{\fsinv{\rk{h^\ad}}(s)\eu{k}}}{h\in\Dt\text{ and }  \int h^\ad W^\mpi h\dif\lambda\leq1}\).
 \item[Step 5:] For \(s\in S\) and \(k\in\mn{1}{q}\), set \(\Dtsk\defeq\setaa{h\in\Dt}{\fsinv{\rk{h^\ad}}(s)\eu{k}=1}\).
 Introducing two positive homogeneous and non-negative functionals \(\Ft\) and \(\Gt\) on \(\Dt\) by \(\Ft(h)\defeq\ek{\int h^\ad W^\mpi h\dif\lambda}^{1/2}\) and \(\Gt(h)\defeq\abs{\fsinv{\rk{h^\ad}}(s)\eu{k}}\), \(h\in\Dt\), similarly to the proof of \rthm{T5.3} one can derive the following assertion.
\end{description}

\begin{thm}\label{T5.5}
 Let \(q\in\N\) and \(s\in S\).
 Then
 \begin{equation}\label{E5.3}
  d
  =\inf\setaa*{\noq{\chuu{s}{k}-t}}{t\in\T{\GsS}}
  =\sup\setaa*{\ek*{\int h^\ad W^\mpi h\dif\lambda}^{-1/2}}{h\in\Dtsk},
 \end{equation}
 where the right-hand side of \eqref{E5.3} is to be interpreted as \(0\) if \(\Dtsk\) is empty.
\end{thm}

\begin{cor}[\tcf{}~\zitaa{MR0426126}{\ccor{3.16}}]\label{C5.6}
 Let \(q\in\N\).
 The set \(\T{\GsS}\) is dense in \(\Loa{2}{W}\) if and only if there does not exist a function \(h\in\Dt\setminus\set{0}\) such that \(\int h^\ad W^\mpi h\dif\lambda<\infty\).
\end{cor}

\bibliography{trigappr_arxiv}
\bibliographystyle{bababbrv}

\vfill\noindent
\begin{minipage}{0.5\textwidth}
 Universit\"at Leipzig\\
 Fakult\"at f\"ur Mathematik und Informatik\\
 PF~10~09~20\\
 D-04009~Leipzig
\end{minipage}
\begin{minipage}{0.49\textwidth}
 \begin{flushright}
  \texttt{
   klotz@math.uni-leipzig.de\\
   maedler@math.uni-leipzig.de
  } 
 \end{flushright}
\end{minipage}

\end{document}